\documentclass[a4paper,12pt]{amsart}

\usepackage{amsfonts}
\usepackage{amsmath}
\usepackage{amssymb}
\usepackage{graphicx}

\usepackage[usenames]{color}
\usepackage[colorlinks]{hyperref}

\newtheorem{defi}{Definition}[section]

\newtheorem{teo}{Theorem}[section]

\newtheorem{remark}{Remark}[section]

\begin{document}
	
	\title[Subelliptic Wave Equations with Log-Lipschitz coefficients]
	{Subelliptic Wave Equations With Log-Lipschitz Propagation Speeds  }
	
	\author[Carlos A.  Rodriguez T.]{ Carlos A. Rodriguez T.}
	\address{Department of Mathematics\\
		Universidad de los Andes\\
		Colombia}
	\email{ca.rodriguez14@uniandes.edu.co}
	\author[Michael Ruzhansky]{Michael Ruzhansky}
	\address{Department of Mathematics: Analysis, Logic and Discrete Mathematics, Ghent University, Belgium \\ and
		School of Mathematical Sciences, Queen Mary University of London, United Kingdom
	}
	\email{michael.ruzhansky@ugent.be}
	\thanks{ The second author was supported by the EPSRC Grant EP/R003025/1, by the Leverhulme Research Grant RPG-2017-151, and by the FWO Odysseus grant G.0H94.18N: Analysis and Partial Differential Equations.}

	\subjclass[2020 MSC]{35G10; 35L30; 22E30}
	
	\keywords{ Fourier Analysis, Lie Groups, Graded Lie Groups, Wave Equation. }

	\date{\today}
	\begin{abstract}
		In this paper we study the Cauchy problem for the wave equations for sums of squares of left invariant vector fields on compact Lie groups and also for hypoelliptic homogeneous left-invariant differential operators on graded Lie groups (the positive Rockland operators), when the time-dependent propagation speed satisfies a Log-Lipschitz condition. We prove the well-posedness in the  associated Sobolev spaces  exhibiting a finite loss of regularity with respect to the initial data, which is not true when the propagation speed is a ${\rm H\ddot{o}lder}$ function. We also indicate an extension to general Hilbert spaces. In the special case of the Laplacian on $\mathbb R^n$, the results boil down to the celebrated result of Colombini-De Giorgi and Spagnolo.
		
	\end{abstract}
	
	\maketitle
	\tableofcontents
	
	\section {Introduction}\label{introduction}
	In this paper we study the well-posedness of a Cauchy problem in two settings, on compact Lie groups and on graded Lie groups. In Section \ref{compact lie group case} we deal with  the problem

	\begin{equation}\label{Cprob1}
	\begin{cases} 
	\partial^2 _t u(t,x)-a(t)\mathcal{L}u(t,x)&=0,\quad \quad (t,x)\in [0,T]\times G, \\
	u(0,x)&=u_0(x), \quad \quad x\in G,\\ 
	\partial_tu(0,x)&= u_1(x), \quad \quad x\in G; 
	\end{cases}
	\end{equation}

	where $\mathcal{L}=X_1^2+X_2^2+...+X_k^2$, $1\leq k\leq  {\rm dim}(G)=n$, is a second order operator which is the sum of squares of elements of the Lie algebra of $G$, namely $X_1$,...,$X_k$, satisfying H${\rm{\ddot{o}}}$rmander condition  of order $l\in \mathbb{N}$. The coefficient function $a:[0,T]\to \mathbb{R}$ is a Log-Lipschitz function, i.e. a function that satisfies
	\begin{equation}\label{loglip}
	|a(t)-a(s)|\leq C|t-s||\log(t-s)|,
	\end{equation}
	for some constant $C>0$ and for all
	$t,s\in[0,T]$.
	We also assume that $a(t)>a_0>0$. 
	
	In   \cite[Theorem 2.3]{GRwave}, it was shown that if $a(t)\geq a_0>0$ and also $a\in C^{\alpha}([0,T])$ is $\rm{H\ddot{o}lder}$ with index $0<\alpha<1 $,  then the Problem \eqref{Cprob1} has a unique solution $u\in C^2([0,T],\gamma^s_{\mathcal{L}}(G))$ provided 
	that 
	\begin{center}
		$u_0,u_1\in \gamma^s_{\mathcal{L}}(G)$ and $1\leq s<1+\frac{\alpha}{1-\alpha},$
	\end{center}
	where $\gamma^s_{\mathcal{L}}(G)$ are the Gevrey spaces of $G$, based on the sub-Laplacian $\mathcal{L}$.
	
	In our work we assume the Log-Lipschitz condition  on the coefficient $a(t)$ which is stronger than the ${\rm H\ddot{o}lder}$ condition  $C^\alpha([0,T])$ with $0<\alpha<1 $, since any Log-Lipschitz function satisfying \eqref{loglip} belongs also to any $C^\alpha([0,T])$ with $0<\alpha<1 $. 
	Nevertheless, there is a physical motivation to study Log-Lipschitz-type functions since they appear in relation to the well-posedness of the Navier-Stokes equations. For instance, consider the solution $u=u(t,x)$ for the problem analysed by Hantaek Bae and Marco Cannone in \cite{HC}, to the problem
	
	$$\begin{cases}
	\partial_t u- \Delta u+(u\cdot\nabla)u+\nabla p&=0,\\
	\nabla\cdot u &=0,\\
	u(0,x)&= u_0(x),
	\end{cases}$$
	
	for $x\in \mathbb{R}^3$, where $u(t,x)$ is the velocity vector field, and $p(t,x)$ is the  scalar pressure  function.
	
	The authors, in  \cite[Theorem 1.1]{HC}, establish  the existence of some $\epsilon>0$ such that  for all $u_0\in H^{\frac{1}{2}}$ with $\|u_0\|_{H^{\frac{1}{2}}}<\epsilon$, there exists a global in time solution $u$ satisfying the Log-Lipschitz  regularity estimate  
	$$\|u\|_{LL_\beta}\leq C_\beta\left(\|u_0\|_{L^1}+\|u_0\|_{H^{\frac{1}{2}}}\right),$$
	where $\displaystyle \|u\|_{LL_\beta}:=\int_0^\infty\sup_{|x-y|<\frac{1}{2}}\frac{|f(t,x)-f(t,y)|}{|x-y|(-|\log|x-y|)^{\beta}}dt,$ for $\beta>0$.

	If we define $\mathcal{H}^s_{\mathcal{L}}(G):=\{u: \|(I-\mathcal{L})^{\frac{s}{2}}u\|_{L^2(G)}<\infty\}$, we prove the following 
	
	\begin{teo}\label{teop}
		Let $a:[0,T]\to \mathbb{R}$ be a Log-Lipschitz function such that $a(t)\geq a_0>0$. Suppose  $(u_0,u_1)\in H^{s}_{\mathcal{L}}(G)\times H^{s-1}_{\mathcal{L}}(G)$ for some $\nu\in \mathbb{R}$. Then the Cauchy problem \eqref{Cprob1} has a unique solution satisfying
		$$ \|u(t,\cdot)\|_{H^{s-\frac{\delta}{2} T}_{\mathcal{L}}}^2+\|\partial_t u(t,\cdot)\|_{H^{s-\frac{\delta}{2} T-1}_{\mathcal{L}}}^2\leq C( \|u_0\|_{H^{s}_{\mathcal{L}}}^2+ \|u_1\|_{H^{s-1}_{\mathcal{L}}}^2),$$
		for some $C, \delta >0$ independent of $u_0, u_1,$ and $t\in [0,T]$.
	\end{teo}

	For the proof of  Theorem \ref{teop} we use the techniques developed in \cite{GRwave} and  \cite{GR}. As in those papers the global Fourier analysis on compact Lie groups introduced in \cite{RT} plays a key role in our work. This and classical results of well-posedness of ordinary first order differential equations will allow us to proof our result.

	In Section \ref{graded lie group case} we study the problem
	\begin{equation}\label{Cprob2}
	\begin{cases}
	\partial^2 _t u(t,x)+a(t)\mathcal{R}u(t,x)&=0,\quad \quad
	(t,x)\in [0,T]\times G, \\
	u(0,x)&=u_0(x),\quad \quad x\in G, \\
	\partial_tu(0,x)&= u_1(x),\quad \quad x\in G,
	\end{cases}
	\end{equation}
	where $G$ is a graded Lie group and $\mathcal{R}$ is a  positive self-adjoint Rockland operator. To analyse the well-posedness of this problem we follow the lines   in \cite{RTar}, and also   in \cite{NR}. The reader should note that in the case of $G = \mathbb{R}$
	and $\mathcal{R} =-\Delta $ , we are dealing with the classical wave equation 
	with the time-dependent propagation speed $a(t)$.  In  \cite{CSM} the authors study the Cauchy problem for strictly hyperbolic operators
	with low regularity coefficients in any space dimension $n\geq 1$. In particular  the
	coefficients of the differential operator are supposed  to be Log-Zygmund continuous in time and Log-Lipschitz continuous in
	space.
	
	The well-posedness results for H${\rm \ddot{o}}$lder
	regular functions $a(t)$  have been obtained by Colombini, de Giorgi and Spagnolo in \cite{CDS}. Moreover, it has been shown by Colombini and Spagnolo in \cite{CS} that already in the case of $G = \mathbb{R}$, 
	the Cauchy problem \eqref{Cprob2} does not have to be well-posed in $C^\infty(\mathbb{R})$.
	
	The Fourier analysis in the case of graded Lie a groups can be found in \cite{FRNILP} and references therein. Also a treatment of $L^p$ estimates for pseudo-differential operators on graded Lie groups can be found in \cite{CDR}.
	The technique used is quite similar to the case of compact Lie groups, but with some differences.

	\section{Compact Lie groups}\label{compact lie group case}
	For a compact Lie group $G$, we denote by $\widehat{G}$ the unitary dual of $G$, consisting of equivalence classes  $[\xi]$ of continuous irreducible unitary representations $\xi:G\to \mathbb{C}^{d_\xi\times d_\xi}$.
	Let $f\in C^\infty(G)$ be a smooth function, we define its Fourier coefficient at $[\xi]\in \widehat{G}$ by $$\widehat{f}(\xi):=\int_{G}f(x)\xi(x)^{*}dx.$$
	Then we have that $$f(x)=\sum_{[\xi]\in \widehat{G}} d_{\xi} {\rm Tr}(\xi(x)\widehat{f}(\xi)),$$
	and $$\|f\|_{L^2(G)}=\left(\sum_{[\xi]\in \widehat{G}} d_{\xi} \|\widehat{f}(\xi)\|^2_{HS}\right)^2,$$ 
	where $\|\widehat{f}(\xi)\|_{HS}:={\rm{Tr}}(\widehat{f}(\xi)\widehat{f}(\xi)^{*})^{\frac{1}{2}}$ is the Hilbert-Schmidt norm.

	For a linear operator $$T:C^\infty(G)\to C^\infty(G),$$ define its global symbol by $$\sigma_T(x,\xi):=\xi^{*}(x)(T\xi)(x)\in \mathbb{C}^{d_\xi\times d_\xi},$$ where  $$[(T\xi)(x)]_{ij}:=T(\xi(x)_{ij}).$$
	
	Using this symbol the following global quantization holds:
	$$Tf(x)=\sum_{[\xi]\in \widehat{G}} d_\xi {\rm Tr}(\xi(x)\sigma_T(x,\xi)\widehat{f}(\xi)).$$
	
	The corresponding symbolic calculus was introduced in \cite{RT}.  Since $\mathcal{L}$ is formally self-adjoint, the symbol of the operator can be diagonalised by a choice of a suitable basis in representation spaces, and its symbol has  constant entries with respect to $x$-variable,  $$\sigma_{-\mathcal{L}}(\xi)={\rm Diag}(\nu_1^2(\xi),..., \nu_{d_\xi}^2(\xi)).$$
	
	Furthermore, the global Fourier analysis permits to characterise the spaces of smooth functions $C^\infty(G)$, the Gevrey spaces associated to the operator $\mathcal{L}$  denoted by $\gamma_{\mathcal{L}}^s(G)$, and the  Sobolev spaces  $\mathcal{H}_{\mathcal{L}}^s(G)$  by:
	
	\begin{equation}
	f\in C^\infty(G) 
	\iff \forall N \, \exists C_N \, {\rm such \, that}\; \|\widehat{f}(\xi)\|_{HS}\leq C_N \langle \xi\rangle^{-N}\; \forall [\xi]\in \widehat{G},
	\end{equation}
	
	\begin{equation}
	f\in \gamma_{\mathcal{L}}^s(G) \iff \exists A>0:\sum_{[\xi]\in \widehat{G}}d_{\xi}\sum_{j=1}^{d_\xi}e^{A\nu_j(\xi)^{\frac{1}{s}}} \left(\sum_{m=1}^{d_\xi}|\widehat{f}(\xi)_{jm}|^2\right)<\infty,
	\end{equation}
	
	\begin{equation}
	f\in \mathcal{H}_{\mathcal{L}}^s(G) \iff \sum_{[\xi]\in \widehat{G}}d_{\xi}\sum_{j=1}^{d_\xi}\left( 1+\nu _j ^2(\xi)\right)^s \left(\sum_{m=1}^{d_\xi}|\widehat{f}(\xi)_{jm}|^2\right)<\infty .
	\end{equation}

	Now we proceed to study the well-posedness of the initial value Problem \eqref{Cprob1}.
	The main idea is to apply the Fourier transform to both sides of the differential equation and then to reduce to a first system which can be analysed by the energy method.
	
	We have 
	\begin{equation}\label{EqFou}
	\partial_t^2\widehat{u}(t,\xi)+a(t)\sigma_\mathcal{-L}(\xi)\widehat{u}(t,\xi)=0,
	\end{equation}
	for any $[\xi]\in \widehat{G}$ fixed.
	In matrix components, the equation (\ref{EqFou}) can be written as
	\begin{equation}\label{EqFou2}
	\partial_t^2\widehat{u}_{mk}(t,\xi)+a(t)\nu_{m}^2(\xi)\widehat{u}_{mk}(t,\xi)=0,
	\end{equation}
	for $1\leq  m,k\leq d_\xi$.
	
	It is then natural to analyse the problem 
	\begin{equation}\label{EqFou3}
	\partial_t^2 \widehat{v}(t,\xi)+a(t)|\xi|_{\nu}^2\widehat{v}(t,\xi)=0,
	\end{equation}
	where for simplicity we denote $|\xi|_\nu:=\nu_m(\xi)$.
	Using the transformation \[V=\begin{pmatrix}
	V_1\\
	V_2
	\end{pmatrix}=\begin{pmatrix}
	i|\xi|_{\nu}\widehat{v}\\
	\partial_t\widehat{v}
	\end{pmatrix}  \]
	and taking \[A=\begin{pmatrix}
	0&1\\
	a(t)&0
	\end{pmatrix},\]  
	we obtain the first order linear differential equation
	\begin{equation}\label{1order}
	\partial_t V(t,\xi)=i|\xi|_\nu A(t)V(t,\xi)
	\end{equation}
	and the initial condition 
	\[V(0,\xi)=\begin{pmatrix}
	i|\xi|_{\nu}\widehat{v_0}(\xi)\\
	\widehat{v_1}(\xi)
	\end{pmatrix} .\]
	
	Now we look for a solution in the form $$V(t,\xi)=\frac{1}{\det(H(t))}e^{-\rho(t)\log|\xi|_\nu }H(t)W(t,\xi),$$ where $\rho\in C^{1}([0,T])$ is a real-valued function to be chosen later. Also $W=W(t,\xi)$ is to be determined.
	Take $\psi\in C^\infty_{c}(\mathbb{R})$, $\psi\geq 0$,  $\int_\mathbb{R}\psi=1$, and $\psi_\epsilon(t):=\frac{1}{\epsilon}\psi(\frac{t}{\epsilon})$. We take $H=H(t)$ to be
	\[H(t)=\begin{pmatrix}
	1&1\\
	\lambda_1(t)& \lambda_2(t)
	
	\end{pmatrix},\]
	
	with $\lambda_1(t)=:(-\sqrt{a}* \psi_\epsilon)$ and $\lambda_2(t)=:(\sqrt{a}* \psi_\epsilon)$, the mollified approximations for the coefficient $a(t)$.
	
	By substitution of  this in equation (\ref{1order}), we obtain that 
	
	$$e^{-\rho(t)\log|\xi|_\nu}(\det H)^{-1}H\partial_tW+e^{-\rho(t)\log|\xi|_\nu}(-\rho^{\prime}(t)\log(|\xi|_\nu))(\det H)^{-1}HW$$

	$$-e^{-\rho(t)\log|\xi|_\nu}\left(\frac{\partial_t \det H}{(\det H)^2}\right)HW+e^{-\rho(t)\log|\xi|_\nu}(\det H)^{-1}(\partial_t H)W$$
	
	$$=i|\xi|_\nu e^{-\rho(t)\log|\xi|_\nu}(\det H)^{-1}AHW.$$
	Multiplying both sides by $e^{\rho(t)\log|\xi|_\nu}(\det H)H^{-1}$ we obtain
	$$\partial_t W-\rho^\prime(t)\log|\xi|_\nu W-\frac{\partial_t \det H}{\det H}W+H^{-1}(\partial_t H)W=i|\xi|_\nu H^{-1}AHW.$$
	
	Now,
	$$\partial_t|W(t,\xi)|^2=2Re\langle\partial_t W(t,\xi),W(t,\xi)\rangle$$ 
	$$ = 2\rho^{\prime}(t)\log|\xi|_\nu|W(t,\xi)|^2+2\left(\frac{\partial_t \det H}{\det H}\right) |W(t,\xi)|^2$$ $$-2 H^{-1}\partial_t H|W(t,\xi)|^2 -|\xi|_\nu \langle H^{-1}AHW-(H^{-1}AH)^{*}W\rangle.$$
	Then, $$\left|\partial_t|W(t,\xi)\right|^2|=|2Re\langle\partial_t W(t,\xi),W(t,\xi)\rangle|$$ 
	$$ \leq 2\Big(|\rho^{\prime}(t)||\log|\xi|_\nu|+\left|\frac{\partial_t \det H}{\det H}\right|  \|H^{-1}\partial_t H\| $$ $$+|\xi|_\nu \| H^{-1}AHW-(H^{-1}AH)^{*}W\|\Big)W(t,\xi)|^2.$$
	We estimate, taking into account that the coefficient function $a=a(t)$ is Log-Lipschitz and $\sqrt{a(t)}\geq a_0>0$ for some $a_0$, the following terms:
	
	\begin{itemize}
		\item[(i)]$\frac{\partial_t \det H}{\det H}$
		\item[(ii)] $ \|H^{-1}\partial_t H\|$
		\item[(iii)] $\|H^{-1}AH-(H^{-1}AH)^{*}\|$
	\end{itemize}

	Observe that $\left|\frac{\partial_t \det H}{\det H}\right|=\left|\frac{\lambda_2^{\prime}(t)-\lambda_1^{\prime}(t)}{\lambda_2(t)-\lambda_1(t)}\right|\leq \frac{1}{2a_0}|\lambda_2^{\prime}(t)-\lambda_1^{\prime}(t)|$
	since $$|\lambda_2(t)-\lambda_1(t)|=|2\int_\mathbb{R}\sqrt{a(\tau)}\psi(\frac{t-\tau}{\epsilon})\epsilon^{-1}d\tau|\geq 2a_0.$$
	
	On the other hand $$|\lambda_2^{\prime}(t)-\lambda_1^{\prime}(t)|=2|\int_\mathbb{R}\sqrt{a(\tau)}\psi^{\prime}(\frac{t-\tau}{\epsilon})\epsilon^{-2}d\tau|$$ $$=2\epsilon^{-1}|\int_{\mathbb{R}}\sqrt{a(t-\epsilon s)}\psi^{\prime}(s)
	ds|$$ $$=2| \epsilon^{-1}\int_\mathbb{R}\left(\sqrt{a(t-\epsilon s)}-\sqrt{a(t)}\right)\psi^{\prime}(s)ds|.$$
	
	Notice that $$|\sqrt{a(t-\epsilon s)}-\sqrt{a(t)}|=\frac{|a(t-\epsilon s)-a(t)|}{\sqrt{a(t-\epsilon s)}+\sqrt{a(t)}}\leq \frac{1}{2a_0} \epsilon |s\log  (\epsilon s)| ,$$ and also that $\int_{\mathbb{R}}\sqrt{a(t)}\psi^{\prime}(s)ds=0$ since $\psi$ has compact support.
	
	We conclude that 
	\begin{equation}\label{i}
	\left|\frac{\partial_t \det H}{\det H}\right|\leq M_1|\log(\epsilon)|
	\end{equation}
	for $M_1>0$ constant. For example, if we take the support of the function $\psi(s)$ to be ${\rm supp}\, \psi\subset [1,2]$, and take $0<\epsilon<\frac{1}{4}$, then $$M_1\leq \frac{1}{a_0}\int_{1}^2|s\psi^\prime (s)|ds,$$
	because $|\log(\epsilon s)|<|\log \epsilon|$ for $\epsilon$ and $s$ satisfying the condition.
	
	To estimate $\|H^{-1}\partial_t H\|$, we compute directly:
	\[H^{-1}\partial_t H= \frac{1}{\lambda_2(t)-\lambda_1(t)}\begin{pmatrix}
	-\lambda_1^{\prime}(t)&-\lambda_2^{\prime}(t)\\
	\lambda_1^{\prime}(t)& \lambda_2^{\prime}(t)
	\end{pmatrix}.\]
	Now, we see that the matrix $H^{-1}\partial_t H$ is symmetric whose eigenvalues are $\beta_1=0$ and $\beta_2=2\lambda_2^\prime(t)$. Then we have $\|H^{-1}\partial_t H\|=\frac{1}{a_0}\lambda_2^\prime(t),$
	from which we obtain
	\begin{equation}\label{ii}
	\|H^{-1}\partial_t H\|\leq M_2|\log(\epsilon)|
	\end{equation}
	for $M_2=M_1=\frac{1}{a_0}\int_{1}^2|s\psi^\prime (s)|ds$.
	
	Now we compute $H^{-1}AH-(H^{-1}AH)^{*}$ 
	\[=\frac{1}{\lambda_2(t)-\lambda_1(t)}\begin{pmatrix}
	0&\lambda_2^{2}(t)+\lambda_1^{2}(t)-2a(t)\\ 
	-\lambda_2^{2}(t)-\lambda_1^{2}(t)+2a(t)&0
	\end{pmatrix}.\]
	We will estimate for $i=1,2$, $$|\lambda_i^2(t)-a(t)|=|\left(\epsilon^{-1}\int_{\mathbb{R}}\sqrt{a(s)}\psi(\frac{t-s}{\epsilon})ds\right)^2-a(t)|.$$ 
	We observe that
	$$|\left(\epsilon^{-1}\int_{\mathbb{R}}\sqrt{a(s)}\psi(\frac{t-s}{\epsilon})ds\right)^2-a(t)|$$ $$=|\left(\epsilon^{-1}\int_{\mathbb{R}}\sqrt{a(s)}\psi(\frac{t-s}{\epsilon})ds\right)^2-\left(\epsilon^{-1}\int_\mathbb{R}\sqrt{a(t)}\psi(\frac{t-s}{\epsilon})ds\right)^2|=$$
	
	$$|\left(\epsilon^{-1}\int_{\mathbb{R}}(\sqrt{a(s)}-\sqrt{ a(t)})\psi(\frac{t-s}{\epsilon})ds\right)\left(\epsilon^{-1}\int_{\mathbb{R}}(\sqrt{a(s)}+\sqrt{ a(t)})\psi(\frac{t-s}{\epsilon})ds\right)|.$$
	It is clear that $$\left(\epsilon^{-1}\int_{\mathbb{R}}(\sqrt{a(s)}+\sqrt{ a(t)})\psi(\frac{t-s}{\epsilon})ds\right)|\leq 2\|\sqrt{a(\cdot)}\|_\infty.|\epsilon^{-1}\int_{\mathbb{R}}\psi(\frac{t-s}{\epsilon})ds|$$ $$=2\|\sqrt{a(\cdot)}\|_\infty.$$
	
	Nevertheless, the first factor allows us to obtain another bound depending on $\epsilon$, in fact $$|\left(\epsilon^{-1}\int_{\mathbb{R}}(\sqrt{a(s)}-\sqrt{ a(t)})\psi(\frac{t-s}{\epsilon})ds\right)|$$
	$$=|\int_\mathbb{R}(\sqrt{a(t-\epsilon\tau)}-\sqrt{a(t)} )\psi(\tau)d\tau|\leq  \frac{1}{2a_0}\left(\int_{1}^2\tau\psi(\tau)d\tau\right) \epsilon |\log(\epsilon)|,$$
	if $\psi$ is as chosen above.
	Finally, taking into account the zeros in the anti-diagonal of the matrix we have that, the norm of the matrix  as an operator is precisely  $2|\lambda_2^2-a^2(t)|$. For this reason,
	\begin{equation}\label{iii}
	\|H^{-1}AH-(H^{-1}AH)^{*}\|\leq M_3\epsilon| \log(\epsilon)|
	\end{equation}
	for $M_3=\frac{1}{2a_0^2}\left(\int_{1}^2\tau\psi(\tau)d\tau\right)$.
	Applying the estimates (\ref{i}), (\ref{ii}), and (\ref{iii}) we obtain
	$$\partial_t|W(t,\xi)|^2\leq$$ $$ 2\left(\rho^{\prime}(t)\log|\xi|_{\nu}+M_1|\log(\epsilon)| +M_2|\log(\epsilon)|+M_3\epsilon|\xi|_\nu|\log(\epsilon)|\right)|W(t,\xi)|^2.$$
	By choosing $\epsilon:=|\xi|_\nu^{-1}<\frac{1}{2}$, we get 
	
	$$\partial_t|W(t,\xi)|^2\leq 2\left(\rho^{\prime}(t)\log|\xi|_{\nu}+(M_1+M_2+M_3)\log|\xi|_\nu\right)|W(t,\xi)|^2.$$
	
	We have, for $|\xi|_\nu \geq 1$, that $$\rho^{\prime}(t)\log|\xi|_{\nu}+(M_1+M_2+M_3)\log|\xi|_\nu\leq 0$$ provided that $\delta>M_1+M_2+M_3$  where
	we assume $\rho(t)=\rho(0)-\delta t$.
	
	We deduce that $\partial_t |W(t,\xi)|^2\leq 0$ under the conditions
	$|\xi_\nu|>2,$ and 
	\begin{equation}\label{delta}
	\delta>M_1+M_2+M_3=2\frac{1}{a_0}\int_{1}^2|\tau\psi^\prime (\tau)|d\tau+\frac{1}{2a_0^2}\left(\int_{1}^2\tau\psi(\tau)d\tau\right).
	\end{equation}
	This implies that $$|V(t,\xi)|=\exp(-\rho(t)\log|\xi|_\nu)\frac{1}{\det(H(t))}\|H(t)\||W(t,\xi)|$$ $$\leq \exp(-\rho(t)|\log|\xi|_\nu)\frac{1}{\det(H(t))}\|H(t)\||W(0,\xi)|$$
	
	$$=\exp((-\rho(t)+\rho(0))\log|\xi|_\nu)\frac{|\det(H(0))|}{|\det(H(t))|}\|H(t)\| \|H^{-1}(0)\||V(0,\xi)|.$$
	
	From this, we have that 
	\begin{equation}\label{Bound}
	|V(t,\xi)|\leq M_4|\xi_\nu|^{\delta T}|V(0,\xi)|
	\end{equation}
	for  $M_4>0 $.
	This implies that $$|\xi|_\nu^2|\widehat{v}|^2+|\partial_t \widehat{v}|^2\leq M_4^2|\xi|_\nu^{2\delta T}\left(|\xi|_\nu^2|\widehat{v_0}|^2+| \widehat{v_1}|^2\right)$$
	
	Coming back to the functions $u_{mk}$, we obtain
	
	\begin{equation}\label{estimate}
	\nu_m^2|\widehat{u}_{mk}|^2+|\partial_t\widehat{u}_{mk}|^2\leq M_4^2\nu_m^{\delta T}\left(\nu_m^2|\widehat{u}_{0mk}|^2+|\widehat{u}_{1mk}|^2\right).
	\end{equation}
	Multiplying both sides by $\nu_m^{2s -2-\delta T}$ we get
	\begin{equation}
	\nu_m^{2s-\delta T}|\widehat{u}_{mk}|^2+\nu_m^{2s -2-\delta T}|\partial_t\widehat{u}_{mk}|^2\leq M_4^2\left(\nu_m^{2s}|\widehat{u}_{0mk}|^2+ \nu_m^{2s-2}|\widehat{u}_{1mk}|^2\right),
	\end{equation}
	which says that
	$$ \|u(t,\cdot)\|_{H^{s-\frac{\delta}{2} T}_{\mathcal{L}}}^2+\|\partial_t u(t,\cdot)\|_{H^{s-\frac{\delta}{2} T-1}_{\mathcal{L}}}\leq C( \|u_0(t,\cdot)\|_{H^{s}_{\mathcal{L}}}^2)+ \|u_1(t,\cdot)\|_{H^{s-1}_{\mathcal{L}}}^2 ,$$
	
	proving Theorem 1.1.
	

	\section{Graded Lie groups}\label{graded lie group case}
	
	In this section we use the Fourier analysis on graded Lie groups, see e.g. \cite{FRNILP} and \cite{NR}, to analyse the Cauchy Problem \eqref{Cprob2}.
	A Lie group $G$ is called graded if its Lie algebra $\mathfrak{g}$ can be decomposed in the form $$\mathfrak{g}=\bigoplus_{i=1}^{k}\mathfrak{g}_i,$$ such that $[\mathfrak{g}_i,\mathfrak{g}_j]\subset \mathfrak{g}_{i+j}$ and $\mathfrak{g}_{i+j}=\{0\}$ for $i+j>k$.
	
	The gradation induces a homogeneous structure on $\mathfrak{g}$ by the dilations $D_r:={\rm Exp}(A\log r),$ where $A: \mathfrak{g}\to \mathfrak{g}$ is a diagonalisable operator acting by $AX=jX$ for $X\in \mathfrak{g}_j$.
	Notice that each $D_r$ is morphism of the Lie algebra $\mathfrak{g}$ for all $r>0$. 
	With this algebraic structure it is possible to extend   the basics of Fourier analysis in $\mathbb{R}^n$ to graded Lie groups $G$. 
	
	We start with $\pi$  a representation of $G$ on the separable Hilbert space $H_\pi$. A vector $v \in H_\pi$
	is said to be smooth or of type $C^\infty$ if the function
	$$G\ni x \mapsto  \pi(x)v \in H_\pi$$ is smooth. The vector space of smooth vectors of a representation $\pi$ is denoted by $H_\pi^\infty$.
	For a function $f\in S(G)=\{f: f\circ \exp_G \in S(\mathfrak{g})\}$ in the Schwartz space of $G$, the Fourier transform evaluated at $\pi \in \widehat{G}$ is the operator acting on $H_\pi$ defined by $$\mathcal{F}_G(f)(\pi):=\widehat{f}(\pi):=\pi(f)=\int_G f(x)\pi(x)^{*}dx.$$
	Let $\mathfrak{g}$ be the Lie algebra of $G$. For every $X\in\mathfrak{g}$,  $v \in H_\pi^\infty$ smooth, and for a given $\pi \in \widehat{G}$, we recall the definition of the infinitesimal representation 
	$$d\pi(X)v := \lim_{t\to 0}\frac{1}{t}\left(\pi(exp_G(tx))v-v\right),$$
	which is a representation of $\mathfrak{g}$ on $H_\pi^\infty$. 
	
	According to the Poincar\'{e}-Birkhoff-Witt Theorem,
	any left-invariant differential operator T on $G$, can be written in a unique way as a finite sum
	$$T =\sum_{|\alpha|\leq M}c_\alpha X^\alpha$$
	where all but finitely many of the coefficients $c_\alpha\in \mathbb{C}$ are zero and $X^\alpha= X_1 . . . X_{|\alpha|}$, for
	$X_i \in \mathfrak{g}$. This allows one to look at any left-invariant differential operator $T$
	on G as an element of the universal enveloping algebra $U(\mathfrak{g})$ of the Lie algebra of $G$. In this case the symbol of the operator $T$ is the family of infinitesimal representations $$\{d\pi(T):=\pi(T)| \pi \in \widehat{G}\}.$$
	A linear operator $T:C^\infty(G)\rightarrow D^\prime(G)$ is homogeneous of  degree $\nu\in \mathbb{C}$ if for every $r>0$ the equality 
	\begin{equation*}
	T(f\circ D_{r})=r^{\nu}(Tf)\circ D_{r}
	\end{equation*}
	holds for every $f\in \mathcal{D}(G). $
	A Rockland operator is a left-invariant differential operator $\mathcal{R}$ which is homogeneous of positive degree $\nu=\nu_{\mathcal{R}}$ and such that, for every unitary irreducible non-trivial representation $\pi\in \widehat{G},$ $\pi(\mathcal{R})$ is injective on $\mathcal{H}_{\pi}^{\infty};$ $\sigma_{\mathcal{R}}(\pi)=\pi(\mathcal{R})$ is the symbol associated to $\mathcal{R}.$
	
	Hulanicki, Jenkins
	and Ludwig showed in \cite{HJL} that the spectrum of $\pi(\mathcal{R})$, with $\pi \in \widehat{G} \setminus \{1\}$, is discrete
	and lies in $(0, \infty)$. Any Rockland operator is a Fourier multiplier and we have that $$\mathcal{F}(\mathcal{R}f)(\pi)=\pi(\mathcal{R})\widehat{f}(\pi),$$
	where $\pi(\mathcal{R})={\rm Diag}(\pi_1^2, \pi_2^2,...)$ with $\pi_i\in \mathbb{R}^{+}$ for $i\in \mathbb{N}$.

	For every $\pi\in \widehat{G}, $ the Kirillov trace character $\Theta_\pi$ defined by  $$(\Theta_{\pi},f):
	={\rm Tr}(\widehat{f}(\pi)),$$ is a tempered distribution on $S(G)$. The identity
	$f(e_G)=\int\limits_{\widehat{G}}(\Theta_{\pi},f)d\pi$,
	implies the Fourier inversion formula $f=\mathcal{F}_G^{-1}(\widehat{f}),$ where
	\begin{equation*}
	(\mathcal{F}_G^{-1}\sigma)(x):=\int\limits_{\widehat{G}}\textnormal{\textbf{Tr}}(\pi(x)\sigma(\pi))d\pi,\,\,x\in G,\,\,\,\,\mathcal{F}_G^{-1}:S(\widehat{G})\rightarrow S(G),
	\end{equation*} is the inverse Fourier  transform. In this context, the Plancherel theorem takes the form $\Vert f\Vert_{L^2(G)}=\Vert \widehat{f}\Vert_{L^2(\widehat{G})}$,  where  $$L^2(\widehat{G}):=\int\limits_{\widehat{G}}H_\pi\otimes H_{\pi}^*d\mu(\pi),$$ is the Hilbert space endowed with the norm $$\Vert \sigma\Vert_{L^2(\widehat{G})}=(\int_{\widehat{G}}\Vert \sigma(\pi)\Vert_{\textnormal{HS}}^2d\pi)^{\frac{1}{2}},$$ with $d\mu(\pi)$ denoting the Plancherel measure on $\widehat{G}.$

	It can be shown that a Lie group $G$ is graded if and only if there exists a differential Rockland operator on $G.$ If the Rockland operator is formally self-adjoint, then $\mathcal{R}$ and $\pi(\mathcal{R})$ admit self-adjoint extensions on $L^{2}(G)$ and $\mathcal{H}_{\pi},$ respectively.

	\begin{defi}
		Let $G$ be a graded
		Lie group and let $\mathcal{R}$ be a positive Rockland operator of homogeneous degree $\nu$. 
		The Sobolev space $H^s_{\mathcal{R}}(G)$ is the subspace of $S^\prime(G)$ obtained by completion of the Schwartz space $S(G)$ with respect to the Sobolev norm
		
		$$\|f\|_{H^s_{\mathcal{R}}(G)}:=\|(1+\pi(\mathcal{R}))^{\frac{s}{\nu}}f\|_{L^2(G)}$$
	\end{defi}

	Now we have the necessary tools to  study the Problem \eqref{Cprob2}. These spaces have been extensively 
	analysed in \cite{FRNILP} and \cite{FRSOBV}.

	\begin{teo}\label{teo2}
		
		Consider the Problem (\ref{Cprob2}) where $\mathcal{R}$ is  a Rockland operator with homogeneous degree $\nu$. For initial data $(u_0,u_1)\in H^{s}\times H^{s-\frac{\nu}{2}}$, the problem is well posed and the solution $u$ satisfy that
		$$\|u\|_{H^{s-\frac{\nu \delta T}{4} }}^2+\|\partial_tu\|_{H^{s-\frac{\nu \delta T}{4} -\frac{\nu}{2}}}   ^{2} \leq M (\|u_0\|_{H^{s}}^2+\|u_1\|^2_{H^{s-\frac{\nu}{2}}}),$$
		for some  constants $\delta, M>0$.
	\end{teo}
	
	\begin{proof}
		
		We apply the group Fourier transform to both sides of the equation to obtain
		$$\partial^2 _t \widehat{u}(t,\pi)+a(t)\pi(\mathcal{R})\widehat{u}(t,\pi)=0.$$
		
		Following the lines of Section \ref{compact lie group case}, and taking into account the diagonal form of the symbol of the Rockland operator $\mathcal{R}$:
		write $\widehat{u}(t,\pi)=[\widehat{u}_{mk}]$, also for the initial data $\widehat{u}_0(t,\pi)=[\widehat{u_0}_{mk}]$, $\widehat{u}_1(t,\pi)=[\widehat{u_1}_{mk}]$. For the eigenvalues of the symbol of the Rockland Operator $\mathcal{R}$, we write $\pi_m$.
		
		The idea is to analyse the first order linear differential equation \begin{equation}\label{1order}
		\partial_t V(t,\xi)=i|\xi|_\pi A(t)V(t,\xi)
		\end{equation}
		with initial condition 
		\[V(0,\xi)=\begin{pmatrix}
		i|\xi|_{\pi}\widehat{v_0}(\xi)\\
		\widehat{v_1}(\xi)
		\end{pmatrix} .\]
		
		We look for a solution in the form $$V(t,\xi)=\frac{1}{\det(H(t))}e^{-\rho(t)\log|\xi|_\pi }H(t)W(t,\xi),$$ where $\rho\in C^{1}([0,T])$ is a real valued function to be chosen later. For $\psi\in C^\infty_{c}(\mathbb{R})$, $\psi\geq 0$,  $\int_\mathbb{R}\psi=1$, and $\psi_\epsilon(t):=\frac{1}{\epsilon}\psi(\frac{t}{\epsilon})$, we take $H=H(t)$ to be
		\[H(t)=\begin{pmatrix}
		1&1\\
		\lambda_1(t)& \lambda_2(t)
		
		\end{pmatrix},\]
		
		with $\lambda_1(t)=:(-\sqrt{a}* \psi_\epsilon)$ and $\lambda_2(t)=:(\sqrt{a}* \psi_\epsilon)$, the mollified approximations for the coefficient function $a(t)$. Using the estimates from Section \ref{compact lie group case} we arrive to, in complete analogy with equation \eqref{estimate}:

		\begin{equation}\label{estimater}
		\pi_m^2|\widehat{u}_{mk}|^2+|\partial_t\widehat{u}_{mk}|^2\leq M\left(\pi_m^{2+\delta T}|\widehat{u_0}_{mk}|^2+\pi^{\delta T}_m|\widehat{u_1}_{mk}|^2\right),
		\end{equation}
		for some $\delta>0$.
		In order to deduce the estimate for functions in the Sobolev spaces, multiplying by $\pi_m^{\frac{4s}{\nu}-\delta T-2}$ we obtain

		\begin{equation}\label{estimater2}
		\pi_m^{\frac{4s}{\nu}-\delta T}|\widehat{u}_{mk}|^{2} +\pi_m^{\frac{4s}{\nu}-\delta T-2}|\partial_t\widehat{u}_{mk}|^2\leq M\left(\pi_m^{\frac{4s}{\nu}}|\widehat{u_0}_{mk}|^2+\pi^{\frac{4s}{\nu}-2}_m|\widehat{u_1}_{mk}|^2\right).
		\end{equation}
		Recall that  $$\|\mathcal{F}\{(1+\pi(\mathcal{R}))^{\frac{s}{\nu}}u\}\|_{HS}^2=\sum_m (1+\pi_m^2)^\frac{2s}{\nu}\sum_j|\widehat{u}_{mj}|^2.$$
		
		Then we can see that  $u \in H^s_{\mathcal{R}}$ is characterised by  
		
		$$ \int_{\widehat{G}} \sum_m (\pi_m^2)^\frac{2s}{\nu}\sum_j|\widehat{u}_{mj}|^2d\mu(\pi)<\infty.$$ 
		Taking this into account we have that 
		$$\|\pi(\mathcal{R})^{\frac{s-\frac{\delta \nu}{4}T}{\nu}}\widehat{u}\|_{HS}^2+\|\pi(\mathcal{R})^{\frac{s-\frac{\delta \nu}{4}T-\frac{\nu}{2}}{\nu}}\partial_t\widehat{u}\|_{HS}^2$$
		$$\leq M\left(\|\pi(\mathcal{R})^{\frac{s}{\nu}}\widehat{u_0}\|_{HS}^2+\|\pi(\mathcal{R})^{\frac{s-\frac{\nu}{2}}{\nu}}\widehat{u}_1\|_{HS}^2\right).$$
		after integration with respect to the Plancherel 
		measure on $\widehat{G}$ on both sides we obtain the proof of the Theorem.
	\end{proof}
	
	\begin{remark}
		In \cite{RothSTein} and \cite{GR} we can find a result on embedding between the Sobolev spaces: the spaces $\mathcal{H}^s_\mathcal{L}:=\{f| (1-\mathcal{L})^{\frac{s}{2}}f \in L^2(G)\}$ where the operator $\mathcal{L}=X_1^2+...+X_k^2$ is a sum of squares of left invariant vector fields satisfying ${\rm H\ddot{o}rmander}$ condition of length $l$,  and the classical ones  $\mathcal{H}^s:=\{f| (1-\Delta)^{\frac{s}{2}}f\in L^2(G)\}$ associated to the Laplace operator $\Delta$. Indeed, $$\mathcal{H}^{s}\subset \mathcal{H}_{\mathcal{L}}^{s}\subset \mathcal{H}^{\frac{s}{l}}.$$
		
		From this we can deduce the well-posedness assuming data in  $\mathcal{H}_{\mathcal{L}}^{s}$ and obtaining solution in classical Sobolev spaces.
	\end{remark}

	\section{An extension to Hilbert spaces }\label{an extension to Hilbert spaces}
	We observe that the technique applied in the sections above can be used to study the problem \eqref{Cprob1} in the context that $u(t)\in \mathcal{H}$ where $\mathcal{H}$  is a separable Hilbert space. 
	
	Let $(e_j)_{j\in \mathbb{N}} \subset \mathcal{H}$ be an orthonormal basis. We define the Fourier transform of the element $u\in  \mathcal{H}$ by $\widehat{u}(j):=\langle u,e_j \rangle_{ \mathcal{H}}$. Clearly $$u=\sum_{j\in \mathbb{N}} \widehat{u}(j)e_j.$$
	
	Consider an operator $\mathcal{A}:\mathcal{D}(\mathcal{A})\subset\mathcal{H}\to\mathcal{H}$
	which acts on $u$ as 
	\begin{equation}
	\mathcal{A}u=\sum_{j\in \mathbb{N}} \lambda_j^2  \widehat{u}(j)e_j,
	\end{equation}
	
	for a sequence of real numbers $(\lambda_j^2)_{j\in \mathbb{N}}$. Notice that $\widehat{\mathcal{A}u}(j)=\lambda_j^2 \widehat{u}(j)$.
	We proceed to define the induced Sobolev spaces $\mathcal{H}^s_\mathcal{A}$ by:
	\begin{itemize}
		\item $ \mathcal{H}^0_\mathcal{A}:=\mathcal{H}$.
		\item $ \mathcal{H}^s_\mathcal{A}:=\{u\in  \mathcal{H}:\displaystyle \sum_{j\in \mathbb{N}} \lambda_j^{2s}|\widehat{u}(j)|^2<\infty\},$ for $s\in \mathbb{R}$.
	\end{itemize}
	
	Now we have introduced analogous tools of the Fourier analysis. We establish the version of problem \eqref{Cprob1}  in this setting. For $a:[0,T]\to \mathbb{R}$ being a Log-Lipschitz function, consider the problem of finding a function $u:[0,T]\to \mathcal{H}$ such that

	\begin{equation}\label{Cprob3}
	\begin{cases}
	\partial^2 _t u(t)-a(t)\mathcal{A}u(t)&=0,\quad
	t\in [0,T], \\
	u(0)&=u_0,\\
	\partial_tu(0)&= u_1,
	\end{cases}
	\end{equation}
	the initial data is in suitable Sobolev spaces.

	\begin{teo}\label{teop3}
		Let $a:[0,T]\to \mathbb{R}$ be a Log-Lipschitz function such that $a(t)>a_0>0$. Suppose  $(u_0,u_1)\in H^{s}_{\mathcal{A}}(G)\times H^{s-1}_{\mathcal{A}}(G)$ for some $s\in \mathbb{R}$. Then the Cauchy problem \eqref{Cprob3} has a unique solution satisfying
		$$ \|u(t,\cdot)\|_{H^{s-\frac{\delta}{2} T}_{\mathcal{A}}}^2+\|\partial_t u(t,\cdot)\|_{H^{s-\frac{\delta}{2} T-1}_{\mathcal{A}}}\leq C( \|u_0\|_{H^{s}_{\mathcal{A}}}^2+ \|u_1\|_{H^{s-1}_{\mathcal{A}}}^2) ,$$
		for some $C, \delta >0$ independent of $u_0, u_1$, and $t\in [0,T]$.
	\end{teo}
	
	\begin{proof}
		Taking Fourier transform on both sides of the equation, we obtain for each $k\in \mathbb{N}$,
		$$\langle \partial_t^2 u(t), e_k\rangle -a(t)\langle \mathcal{A}u(t),e_k\rangle =0,$$
		or
		$$\partial_t^2\widehat{u(t)}(k)- a(t)\lambda_k^2 \widehat{u(t)}(k)=0.$$  After denoting $\widehat{u(t)}=\beta(t)$, we deal with the equation
		
		$$\partial_t^2 \beta(t)- a(t)\lambda_k^2 \beta(t)=0.$$
		
		This is clearly the equation \eqref{EqFou3}. The conclusions obtained in Theorem \ref{teop} are then valid in this case. 
	\end{proof}

\end{document}